\documentclass{article}

\usepackage{graphicx,xypic}
\usepackage{amsthm}
\usepackage{amsmath,amssymb}
\usepackage{amsfonts}
\usepackage{xcolor}
\usepackage[margin=0.9in]{geometry}
\usepackage[shortlabels]{enumitem}
\usepackage{sectsty}
\usepackage{hyperref}
\usepackage{mathtools}
\usepackage{multirow} 
\usepackage{multicol} 
\usepackage{changepage}
\usepackage{float}
\usepackage{indentfirst}
\usepackage{titlecaps}
\usepackage{amssymb}
\usepackage[nottoc]{tocbibind}
\usepackage{tikz}

\usepackage{MnSymbol}

\usepackage[runin]{abstract}

\usepackage{titlesec}
\usepackage{sectsty}
\sectionfont{\centering} 

\sectionfont{\fontsize{12}{15}\selectfont}
\newlength{\rowA}
\newcommand{\strutA}{
\rule[-0.45\rowA]{0pt}{\rowA}
}

\title{Balancedness of Normal Bundles of Rational Curves in Grassmannians}
\author{An Cao}
\date{}
\theoremstyle{plain}
\newtheorem{thm}{Theorem}[section]
\newtheorem{lemma}[thm]{Lemma}
\newtheorem{conj}[thm]{Conjecture}
\newtheorem{cor}[thm]{Corollary}
\newtheorem{prop}[thm]{Proposition}
\newtheorem{case}{Case}
\theoremstyle{definition}

\newtheorem{definition}[thm]{Definition}

\newcommand{\pdv}[2]{\frac{\partial #1}{\partial #2}}

\begin{document}

\maketitle
\titleformat{\section}[block]{\filcenter\scshape}{\thesection.}{0.5em}{\titlecap}
\titleformat{\subsection}[block]{\scshape}{\thesubsection.}{0.5em}{\titlecap}

\begin{abstract}
In projective space $\mathbb{P}^r$, the normal bundle of a general rational curve of any degree $d \geq r$ is balanced except over fields of characteristic 2. In \cite{CLV24}, Coskun--Larson--Vogt found counterexamples to the naive generalization of this statement to general rational curves in Grassmannians and proposed a conjecture for when the normal bundle is balanced. In this paper, we prove their conjecture for the Grassmannians $G(2, 4)$, $G(2, 5)$, and $G(2, 6)$.
\end{abstract}

\section{Introduction}
Rational curves are a fundamental object of study in algebraic geometry. The normal bundles of rational curves, which capture the geometric properties such as their deformation theory, have been a focus of intense research \cite{AR17, ALY19, AR17, CR18, EV81, EV82, GS80, Ran07, S80, S82, LV23}. In the context of projective space, Coskun and Riedl \cite{CR18} classified the splitting types of normal bundles of rational curves, building on earlier work of Sacchiero \cite{GS80, S80, S82}, which states that the generic such normal bundle can be decomposed into line bundles with degrees being at most 1 apart, a property which we call \textit{balanced}. 

The naive analog of this statement is false in Grassmannians \cite{CLV24}, with exceptions found under certain numerical conditions that led to the following conjecture proposed in \cite{CLV24}, which we now explain.
\begin{itemize}
    \item[(i)] Degeneracy exceptions: If $1 < d < \min(k,n-k)$, the rational curve $C$ lives inside a smaller Grassmannian $G(d,d+\max(k,n-k))$. We have $N_{C/G(k,n)} \simeq N_{C/G(d,d+\max(k,n-k))} \oplus (Q|_C)^{\min(k,n-k)-d}$.
    Calculating the slope, we have $Q|_C$ has a summand with degree $0$, and $N_{C/G(d,d+\max(k,n-k))}$ has a summand with degree at least $2$. Hence, the normal bundle $N_{C/G(k,n)}$ cannot be balanced.
    
    If $\min(k,n-k) < d < \max(k,n-k)$, applying duality $G(k,n) \simeq G(n-k,n)$ gives a similar splitting $N_{C/G(k,n)} \simeq N_{C/G(d,d+\min(k,n-k))} \oplus (S|_C^*)^{\max(k,n-k)-d}$. Calculating the slope of each component shows the normal bundle cannot be balanced.
    \item[(ii)] Tangent bundle splitting exceptions: If we write $d = k q_1 + r_1$ with $0 \leq r_1 < k$ and
$d = (n - k) q_2 + r_2$ with $0 \leq r_2 < n - k$, then $S^*|_C \simeq \mathcal{O}(q_1)^{k-r_1} \oplus \mathcal{O}(q_1+1)^{r_1}$ and $Q|_C \simeq \mathcal{O}(q_2)^{n-k-r_2} \oplus \mathcal{O}(q_2+1)^{r_2}$. Given the tensor structure of the tangent bundle $T_{G(k,n)} \simeq S^* \otimes Q$, we have the splitting type of the restricted tangent bundle $T_{G(k,n)}|_C$ is $\simeq \mathcal{O}(q_1+q_2)^{r_1 r_2} \oplus \mathcal{O}(q_1+q_2+1)^{r_1 (n-k-r_2) + (k-r_1)r_2} \oplus \mathcal{O}(q_1+q_2+2)^{(k-r_1)(n-k-r_2)} $. If $r_1 r_2 \neq 0$, then the normal bundle $N_{C/G(k,n)}\simeq \frac{T_{G(k,n)}|_C}{T_C}$ must have a summand of degree $\geq q_1+q_2+2$, and if $q_1 + q_2 \leq (k - r_1)(n - k - r_2)$, the normal bundle must have a summand of degree $q_1+q_2$. Thus these conditions will prevent the normal bundle from being balanced.
    \item[(iii)] Characteristic 2 exceptions: If $k = 1$, we have $N_C$ is balanced if and only if  $N^*_C(1)$ is balanced. Since $N^*_C(1)$ is pullback of a vector bundle under the Frobenius morphism \cite{CLV22,LV23}, all degrees of its summands must be even, and so they must be equal if $N^*_C(1)$ is balanced. Hence, $N_C$ is balanced only if $\mu(N^*_C(1)) = \frac{2-2d}{n-2} \in 2\mathbb{Z}$, that is $d \equiv 1 \bmod{n-2}$. Since $G(1,n) \simeq G(n-1,n)$, we need the same condition for $k=n-1$.
\end{itemize}

Coskun-Larson-Vogt proposed that these are the only exceptions to the normal bundle being balanced and made the following conjecture.

\begin{conj}[Coskun-Larson-Vogt]\label{conjecture} The normal bundle of a general nondegenerate rational curve of degree $d$ in a Grassmannian $G(k,n)$ is balanced except for the following cases:
\begin{enumerate}[(i)]
    \item $1 < d < \min(k,n-k)$ or $\min(k,n-k) < d < \max(k, n - k)$,
    \item If we write
$d = k q_1 + r_1$ with $0 \leq r_1 < k$ and
$d = (n - k) q_2 + r_2$ with $0 \leq r_2 < n - k$
then $d \neq 1$, $r_1 \neq 0$, $r_2 \neq 0$, and $q_1 + q_2 \leq (k - r_1)(n - k - r_2)$.
\item The characteristic is 2, $k=1$ or $k = n-1$, and $d \not \equiv 1 \bmod{n-2}$.
\end{enumerate}
\end{conj}

The main result of this paper is the following. 

\begin{thm}\label{mainthm}
    Conjecture \ref{conjecture} holds for $G(2,4)$ and $G(2,6)$ in characteristic $0$, and holds for $G(2,5)$ in any characteristic.
\end{thm}

\paragraph{Acknowledgements.}I would like to express my sincere gratitude to Professor Eric Larson for his guidance and support throughout this project. This research was supported by funding from the Brown SPRINT/UTRA program and the NSF grants DMS-2200641 and DMS-2440719.
\section{Methods}
Since the condition of balancedness is open, it suffices to exhibit a particular rational curve of degree $d$ in $G(k,n)$ whose normal bundle is balanced.

A rational curve of degree $d$ in a Grassmannian $G(k,n)$ can be parameterized as the image of the map 
$$ \mathbb{P}^1 \xrightarrow{\text{degree } d} G(k,n)$$
\begin{equation}\label{omatrix}
[x:y] \longmapsto 
\begin{bmatrix}
    f_{11} & f_{12} & f_{13} & \dots  & f_{1n} \\
    f_{21} & f_{22} & f_{23} & \dots  & f_{2n} \\
    \vdots & \vdots & \vdots & \ddots & \vdots \\
    f_{k1} & f_{k2} & f_{k3} & \dots  & f_{kn}
\end{bmatrix}
\end{equation}
where $(f_{i1}(x,y), f_{i2}(x,y), \cdots, f_{in}(x,y))$ is a $n$-tuple of homogeneous polynomials of degree $d_i$ in $x,y$ for $i = \overline{1,k}$ such that $\sum_{i=1}^k d_i = d$ and the matrix has rank $k$ for all $(x,y) \neq (0,0)$.

For such a rational curve, the dual of the universal subbundle is $S^* = \mathcal{O}(d_1) \oplus \mathcal{O}(d_2)\oplus \cdots \oplus  \mathcal{O}(d_k) $. The universal quotient bundle $Q$ corresponds to the module of relations among the columns of the matrix \eqref{omatrix}. Assuming the map is unramified, these are related to the normal bundle via the following exact sequence.
\begin{equation}\nonumber
\begin{aligned}
    0 \rightarrow T_{\mathbb{P}^1} \rightarrow S^* \otimes Q = T_{G(k,n)}|_{\mathbb{P}^1} \rightarrow  N_{\mathbb{P}^1/G(k,n)} \rightarrow 0
\end{aligned}
\end{equation}

In order to study the normal bundle, we will explicitly write down the map $ T_{\mathbb{P}^1} \rightarrow S^* \otimes Q$ in the above sequence. If the module of relations is generated by $g_{lj}$, where $l = \overline{1,n-k}$ so that $\sum_{j=1}^n g_{lj}f_{ij}=0$, then the map $ T_{\mathbb{P}^1} \rightarrow S^* \otimes Q$ is given by the expression  
\begin{equation}\label{par-der}
\begin{aligned}
    \frac{1}{x} \sum g_{lj} \pdv{f_{ij}}{y} = -\frac{1}{y} \sum g_{lj}\pdv{f_{ij}}{x}
\end{aligned}
\end{equation}
and the map $\mathbb{P}^1 \xrightarrow{\text{degree } d} G(k,n)$ is unramified if none of these vanish simultaneously when $(x,y) \neq (0,0)$. 

We will decompose $Q = Q' \oplus Q''$ so that in the exact sequence
\begin{equation}\label{exact-seq}
\begin{aligned}
    0 \rightarrow S^* \otimes Q' \rightarrow N_{\mathbb{P}^1/G(k,n)} \rightarrow  (S^* \otimes Q'') / {T_{\mathbb{P}^1}}  \rightarrow 0,
\end{aligned}
\end{equation}
the tensor product $S^* \otimes Q' = 
\bigoplus \mathcal{O}(a_i)$ where the $a_i$'s have the desired degrees for the balanced splitting type. (Depending on the case specifics, we can also decompose $S$ and obtain variants of this exact sequence.) Then we will find independent relations among the polynomials represented by the map $T_{\mathbb{P}^1} \rightarrow S^* \otimes Q''$. These will show that the degrees of the summands of $S^* \otimes Q'$ and $(S^* \otimes Q'') / {T_{\mathbb{P}^1}}$ are within 1 of each other. The balancedness of the normal bundle then follows from the following lemma.

\begin{lemma}[Lemma 2.1 in \cite{CLV24}]\label{deg-exact-seq}
    Let $0 \rightarrow E \rightarrow F \rightarrow G \rightarrow 0$ be an exact sequence of vector bundles on $\mathbb{P}^1$. If the degrees of all of the summands of $E$ and $G$ lie in a fixed interval, then the degrees of all of the summands of $F$ lie in the same interval.
\end{lemma}

We briefly recall the proof here. Suppose the degrees of all of the summands of $E$ and $G$ lie in $[m,n]$.  Then $h^0(E(-m-1))=h^0(G(-m-1))=0$ and $h^1(E(-n-1)) = h^1(G(-n-1))=0$, so $h^0(F(-m-1))=0$ and $h^1(F(-n-1))=0$. Hence, the degrees of all of the summands of $F$ also lie in $[m,n]$.

\paragraph{Outline.}In the following section, we will prove Theorem \ref{mainthm}. We will separately prove the theorem for $G(2,4)$, $G(2,5)$, and $G(2,6)$ in each of subsections $3.1,$ $3.2$, and $3.3$.

\section{Proof of Theorem 1.2}
\subsection{$G(2,4)$}

\begin{thm}\label{g24thm}
In characteristic $0$, the normal bundle of a general nondegenerate rational curve of any degree $d$ in the Grassmannian $G(2,4)$ is balanced.
\end{thm}

\begin{proof} 
The theorem has been proven for $d = 1$. 
Write $d = 6t + r$, $r \in \{0,1,2,3,4,5\}$. We will divide into the following cases:
\begin{case} $r \neq 5$, $t \geq 1$
\normalfont

Let $a = \lfloor \frac{r}{2}\rfloor$, $b = \lceil \frac{r}{2} \rceil$.
Consider the map
\begin{equation}\label{g24-1}
[x:y] \longmapsto 
\begin{bmatrix}
x^{5t+r-1} + x^{2t+a-1}y^{3t+b} & y^{5t+r-1} & x^{5t+r-2}y & 0\\
y^{t+1} & 0 & x^{t}y+xy^{t} & x^{t+1}
\end{bmatrix}
\end{equation}

The rows of the matrix are linearly independent as the minor of the first 2 columns is a power of $y$, and when plugging in $y=0$, the minor of the first and last columns is a power of $x$.

Let $q = \lfloor \frac{t}{2} \rfloor$. We claim the module of relations between the columns of \eqref{g24-1}, which corresponds to the universal quotient bundle $Q$, is generated by the following $2$ relations of degrees $3t+a$ and $3t+b$:
{\small
\begin{table}[h]
    \centering
    \begin{tabular}{c c c c}
$-x^{t+1}y^{2t+a-1}$ & $x^{3t+a}$ &  $x^{t+2}y^{2t+a-2}$ & $y^{3t+a} - x^2 y^{3t+a-2} - x^{t+1}y^{2t+a-1}$\strutA\\ 
\hline
\multirow{4}{3.4cm}[1ex]{$-x^{3t+b-1}y -x^{t}y^{2t+b} \newline- (xy^{3t+b-1}+x^3 y^{3t+b-3}\newline+\cdots  + x^{2q+1} y^{3t+b-1-2q})$
} &  \multirow{4}{3.6cm}[1ex]{ $x^{3t+a-1}y^{b-a+1} 
\newline+ (x^{2t+a}y^{t+b-a}\newline + x^{2t+a+2}y^{t+b-a-2} + \cdots \newline+ x^{2t+a+2q+2}y^{t+b-a-2q-2})$} & \multirow{4}{3.3cm}[1ex]{$x^{3t+b}+x^{t+1} y^{2t+b-1} \newline+ (y^{3t+b} + x^2y^{3t+b-2}+\cdots \newline + x^{2q+2} y^{3t+b-2-2q})$} &  \multirow{4}{5cm}[1ex]{$-x^{3t+b-1}y - x^{2t+b}y^{t} + x^{2t+b-2}y^{t+2} \newline- x^{t}y^{2t+b} - xy^{3t+b-1} \newline- x^{2q+2-t}y^{4t+b-2q-2} -(xy^{3t+b-1} \newline+x^3 y^{3t+b-3}+\cdots + x^{2q+1} y^{3t+b-1-2q})$}\strutA 
    \end{tabular}
\end{table}
}
\vspace{20pt}

Indeed, it is easy to check that these are relations between the columns of \eqref{g24-1}. Moreover, when plugging in $(x,y)=(1,1)$, the minor of the first two columns is nonzero, hence these relations are independent. Finally, these $2$ relations have the sum of degrees equal to $6t+a+b = d$, so they generate the module of relations. 

Decompose $S = S' \oplus S''$, where $S'$ corresponds to the first row and $S''$ corresponds to the second row. We derive a variant of the exact sequence \eqref{exact-seq}:
\begin{align*}
    0 \rightarrow S'^* \otimes Q \rightarrow N_{\mathbb{P}^1/G(k,n)} \rightarrow  (S''^* \otimes Q) / {T_{\mathbb{P}^1}}  \rightarrow 0
\end{align*}

Using \eqref{par-der}, the map 
$ T_{\mathbb{P}^1} \rightarrow S''^* \otimes Q$ is given by $2$ polynomials of degrees $4t+a-1, 4t+b-1$, respectively, say $q_{12}, q_{22}$.
{\allowdisplaybreaks
\begin{align*}
    q_{12} &= \frac{1}{x}\left\{-x^{t+1}y^{2t+a-1}(t+1)y^{t} + x^{t+2}y^{2t+a-2}(x^{t} + txy^{t-1}) \right\} \\
&= -(t+1)x^{t}y^{3t+a-1} +  x^{2t+1}y^{2t+a-2} + tx^{t+2}y^{3t+a-3} \\
&= x^{t}y^{2t+a-2} \left[x^{t+1} + tx^2y^{t-1}-(t+1)y^{t+1} \right] \\
q_{22} &= \frac{1}{x}\{[-x^{3t+b-1}y -x^{t}y^{2t+b} - (xy^{3t+b-1}+x^3 y^{3t+b-3}+\cdots + x^{2q+1} y^{3t+b-1-2q})](t+1)y^{t} \\
&\qquad+ [x^{3t+b}+x^{t+1} y^{2t+b-1} + (y^{3t+b} + x^2y^{3t+b-2}+\cdots + x^{2q+2} y^{3t+b-2-2q})] \cdot [x^{t} + txy^{t-1}]\} \\
&= - (t+1)x^{3t+b-2}y^{t+1} - (t+1)x^{t-1}y^{3t+b} - (t+1)(y^{4t+b-1} + x^2y^{4t+b-3}+\cdots + x^{2q} y^{4t+b-2q-1}) \\
& \qquad + x^{4t+b-1} + x^{2t}y^{2t+b-1} + (x^{t-1}y^{3t+b} + x^{t+1}y^{3t+b-2} + \cdots + x^{t+2q+1}y^{3t+b-2q-2}) \\
&\qquad +  t(x^{3t+b}y^{t-1} + x^{t+1}y^{3t+b-2}) + t(y^{4t+b-1} + x^2y^{4t+b-3}+\cdots + x^{2q+2} y^{4t+b-2q-3})\\
&= x^{4t+b-1} + tx^{3t+b}y^{t-1} - (t+1)x^{3t+b-2}y^{t+1} + x^{2t}y^{2t+b-1} +  tx^{t+1}y^{3t+b-2} - (t+1)x^{t-1}y^{3t+b} + tx^{2q+2} y^{4t+b-2q-3} \\
& \qquad + (x^{t-1}y^{3t+b} + x^{t+1}y^{3t+b-2} + \cdots + x^{t+2q+1}y^{3t+b-2q-2}) - (y^{4t+b-1} + x^2y^{4t+b-3}+\cdots + x^{2q} y^{4t+b-2q-1}) \\
&= \left[x^{t+1} + tx^2y^{t-1}-(t+1)y^{t+1} \right] \cdot [x^{3t+b-2} + x^{t-1}y^{2t+b-1} + (x^{2q}y^{3t+b-2q-2} + x^{2q-2}y^{3t+b-2q} + \cdots + y^{3t+b-2})] \\
& \qquad + x^{t-1}y^{3t+b} + ty^{4t+b-1}
\end{align*}
}
Since coefficients of $x^{4t+b-1}$ and $y^{4t+b-1}$ in $q_{22}$ are nonzero, we have $\gcd(q_{22},x) = \gcd(q_{22},y) = 1$. Thus
\begin{equation}
\begin{aligned}
\notag
\gcd(q_{12},q_{22}) &= \gcd(x^{t+1} + tx^2y^{t-1}-(t+1)y^{t+1}, q_{22}) \\
&= \gcd(x^{t+1} + tx^2y^{t-1}-(t+1)y^{t+1}, x^{t-1}y^{3t+b} + ty^{4t+b-1}) \\
&= \gcd(x^{t+1} + tx^2y^{t-1}-(t+1)y^{t+1}, x^{t-1} + ty^{t-1}) \\
&= \gcd(-(t+1)y^{t+1}, x^{t-1} + ty^{t-1}) \\
&= 1
\end{aligned}
\end{equation}

 Since $\gcd(q_{12},q_{22}) = 1$, the map \eqref{g24-1} is unramified. We have the following relation of degree $8t+r-2$:
$$q_{12}q_{22} - q_{22}q_{12} = 0.$$

This gives $(S''^* \otimes Q)/T_{\mathbb{P}_1} \simeq \mathcal{O}(8t+r)$. On the other hand, $S'^* \otimes Q\simeq \mathcal{O}(8t+r-1+a) \oplus \mathcal{O}(8t+r-1+b)$. One can easily check that the degrees of the summands of $S'^* \otimes Q$ and $(S''^* \otimes Q)/T_{\mathbb{P}_1}$ are within $1$ of each other, hence so are the degrees of the summands of the normal bundle.
\end{case}

\begin{case} $r = 5$, $t \geq 1$
\normalfont

Consider the map
\begin{equation}\label{g24-2}
[x:y] \longmapsto 
\begin{bmatrix}
x^{5t+3} & y^{5t+3} & x^{3t+4}y^{2t-1}+xy^{5t+2} & xy^{5t+2}\\
y^{t+2} & 0 & x^{t+1}y+x^2 y^{t} & x^{t+2}
\end{bmatrix}
\end{equation}

The rows of the matrix are linearly independent as the minor of the first 2 columns is a power of $y$, and when plugging in $y=0$, the minor of the first and last columns is a power of $x$.

We claim the module of relations between the columns of \eqref{g24-2}, which corresponds to the universal quotient bundle $Q$, is generated by the following relations:
\begin{center}
\begin{tabular}{c c c c}
\multirow{2}{3cm}[1ex]{$-x^{t+4}y^{2t-1} \newline+ x^{t+1}y^{2t+2} + x^2 y^{3t+1}$} &  \multirow{3}{5.4cm}[1ex]{$x^{3t+3} + x^{3t+1}y^2 - x^{3t}y^3 + x^{2t+4}y^{t-1} \newline+ x^{2t+2}y^{t+1} - 2x^{2t+1}y^{t+2} - x^{t+2}y^{2t+1}\newline -x^3 y^{3t} + xy^{3t+2}$} & \multirow{2}{2.6cm}[1ex]{$x^{3t+3} - x^{3t} y^3 \newline- x^{2t+1} y^{t+2} - y^{3t+3}$} &  \multirow{3}{3cm}[1ex]{$-x^{3t+2}y + x^{3t-1}y^4 \newline- x^{2t+3}y^{t} + 2x^{2t}y^{t+3} \newline+ x^{t+1}y^{2t+2} + x^2 y^{3t+1} $}\strutA \\[0.8cm]\hline
$-x^{t+2}y^{2t}$ & $-x^{3t+2}+x^{3t+1}y + x^{2t+2}y^t-xy^{3t+1}$ &  $x^{3t+1}y $ & $-x^{3t} y^2- x^{2t+1}y^{t+1}  + y^{3t+2}$\strutA
\end{tabular}
\end{center}

Indeed, it is easy to check that these are relations between the columns of \eqref{g24-2}. Moreover, when plugging in $(x,y)=(1,1)$, the minor of columns $1,3$ is nonzero, hence these relations are independent. Finally, these $2$ relations have the sum of degrees equal to $d = 6t+5$, so they generate the module of relations.

Decompose $Q = Q' \oplus Q''$, where $Q'$ corresponds to the first row and $Q''$ corresponds to the second row. Decompose $S = S' \oplus S''$, where $S'$ corresponds to the first row and $S''$ corresponds to the second row of the matrix \eqref{g24-2}. We derive a variant of the exact sequence \eqref{exact-seq}:
\begin{align*}
    0 \rightarrow S'^* \otimes Q' \rightarrow N_{\mathbb{P}^1/G(k,n)} \rightarrow  (S''^* \otimes Q' \oplus S^* \otimes Q'') / {T_{\mathbb{P}^1}}  \rightarrow 0
\end{align*}

Using $\eqref{par-der}$, the map $ T_{\mathbb{P}^1} \rightarrow S''^* \otimes Q' \oplus S^* \otimes Q''$ is given by 3 polynomials of degrees $8t+3, 4t+3, 4t+2$, respectively, say $q_{21}, q_{12}, q_{22}$. 
{\allowdisplaybreaks
\begin{align*}
q_{21} &= \frac{1}{y}\{-x^{t+2}y^{2t}(5t+3)x^{5t+2} + x^{3t+1}y\left[(3t+4)x^{3t+3}y^{2t-1} + y^{5t+2} \right] + (-x^{3t}y^2 - x^{2t+1}y^{t+1}+y^{3t+2})y^{5t+2} \}\\
&= -(2t-1)x^{6t+4}y^{2t-1} + x^{3t+1}y^{5t+2} -x^{3t}y^{5t+3} - x^{2t+1}y^{6t+2} + y^{8t+3} \\
q_{22} &= \frac{1}{y}\{ x^{3t+1}y [(t+1)x^ty+2xy^t] + (- x^{3t} y^2- x^{2t+1}y^{t+1} + y^{3t+2})(t+2)x^{t+1}\} \\
&= (t+1)x^{4t+1}y + 2x^{3t+2}y^t - (t+2)x^{4t+1}y - (t+2)x^{3t+2}y^t + (t+2)x^{t+1}y^{3t+1} \\
&= -x^{4t+1}y - tx^{3t+2}y^t + (t+2)x^{t+1}y^{3t+1}  \\
&= -x^{t+1}y \left[x^{3t} + tx^{2t+1}y^{t-1} - (t+2)y^{3t}\right] \\
q_{12} &= \frac{1}{y}\{ (x^{3t+3} - x^{3t} y^3 - x^{2t+1} y^{t+2} - y^{3t+3})[(t+1)x^ty+2xy^t]  \\
&\qquad + (-x^{3t+2}y + x^{3t-1}y^4 - x^{2t+3}y^{t} + 2x^{2t}y^{t+3} + x^{t+1}y^{2t+2} + x^2 y^{3t+1})(t+2)x^{t+1}\}  \\
&= -x\,\big[x^{4t+2}- x^{4t-1}y^3 + tx^{3t+3}y^{t-1} -(t+1)x^{3t}y^{t+2} - tx^{2t+1}y^{2t+1} - (t+2)x^{t+2}y^{3t} + (t+1)x^{t-1} y^{3t+3} + 2y^{4t+2} \big]  \\
&= -x\left[(x^{3t} + tx^{2t+1}y^{t-1} - (t+2)y^{3t}) (x^{t+2}-x^{t-1}y^3-y^{t+2}) - x^{t-1}y^{3t+3}-ty^{4t+2}\right]
\end{align*}
}

We have 
{\allowdisplaybreaks
\begin{align*}
\notag
&\gcd(q_{12},q_{22}) \\&= x \cdot \gcd(x^{3t} + tx^{2t+1}y^{t-1} - (t+2)y^{3t}, (x^{3t} + tx^{2t+1}y^{t-1} - (t+2)y^{3t} )(x^{t+2}-x^{t-1}y^3-y^{t+2}) - x^{t-1}y^{3t+3}-ty^{4t+2} ) \\
&= x \cdot \gcd(x^{3t} + tx^{2t+1}y^{t-1} - (t+2)y^{3t}, \; - x^{t-1}y^{3t+3}-ty^{4t+2} ) \\
&= x \cdot \gcd(x^{3t} + tx^{2t+1}y^{t-1} - (t+2)y^{3t}, \; - (x^{t-1}+ty^{t-1})y^{3t+3} ) \\
&= x \cdot \gcd(x^{3t} + tx^{2t+1}y^{t-1} - (t+2)y^{3t}, \; x^{t-1}+ty^{t-1} ) \\
&= x \cdot \gcd(x^{2t+1}( x^{t-1}+ty^{t-1}) - (t+2)y^{3t}, \; x^{t-1}+ty^{t-1} ) \\
&= x \cdot \gcd(- (t+2)y^{3t}, \; x^{t-1}+ty^{t-1} ) \\
&= x
\end{align*}
}

We have $x \nmid q_{21}$, so the map $\eqref{g24-2}$ is unramified. At $x=0$, the fiber of the map $T_{\mathbb{P}^1} \to S''^* \otimes Q' \oplus S^* \otimes Q''$ is contained in the fiber of $S''^* \otimes Q'$, so we have the exact sequence
\begin{align*}
    0 \rightarrow (S''^* \otimes Q')[x=0] \rightarrow  (S''^* \otimes Q' \oplus S^* \otimes Q'') / {T_{\mathbb{P}^1}} \rightarrow (S^* \otimes Q'') / {T_{\mathbb{P}^1}} \rightarrow 0
\end{align*}

The map $T_{\mathbb{P}^1} \rightarrow S^* \otimes Q''$ is given by $q_{12}$ and $q_{22}$. Write $q_{12} = x q_{12}'$, $q_{22} = x q_{22}'$, then $\gcd(q_{12}',q_{22}') = 1$. We have a relation of degree $8t+4$: $q_{12}'q_{22} - q_{22}'q_{12} = 0.$ This gives $(S^* \otimes Q'')/T_{\mathbb{P}_1} \simeq \mathcal{O}(8t+6)$. We also have $(S''^* \otimes Q')(x)\simeq \mathcal{O}(8t+6)$, so by {Lemma \ref{deg-exact-seq}},  $(S''^* \otimes Q' \oplus S^* \otimes Q'') / {T_{\mathbb{P}^1}} \simeq O(8t+6)^2$. On the other hand, $S'^* \otimes Q'\simeq \mathcal{O}(8t+6)$. Hence, all the degrees of the summands of $S'^* \otimes Q'$ and $(S''^* \otimes Q' \oplus S^* \otimes Q'')/T_{\mathbb{P}_1}$ are equal, hence so are the degrees of the summands of the normal bundle.

\end{case}
\begin{case}
    $d \leq 5$
    \begin{itemize} \normalfont
        \item[(3.1)] $d = 2$:
        Consider the map \begin{equation}\label{g24-d=2}
[x:y] \longmapsto 
\begin{bmatrix}
x & y & 0 & 0 \\
y & 0 & x & y
\end{bmatrix}
\end{equation}
The rows of the matrix are linearly independent as the minor of columns $2,4$ is a power of $y$, and the minor of columns $1,3$ is a power of $x$.

We claim the module of relations between the columns of \eqref{g24-d=2}, which corresponds to the universal quotient bundle $Q$, is generated by the following relations:
$$\begin{tabular}{c c c c}
$-y $& $x$ & 0 & $y$\\\hline
0 & 0 & $-y$ & $x$
\end{tabular}$$
Indeed, it is easy to check that these are relations between the columns of \eqref{g24-d=2}. Moreover, the minor of columns $1,3$ is a power of $y$, and the minor of columns $2,4$ is a power of $x$, hence these relations are independent everywhere, and so they generate the module of relations. 

Using \eqref{par-der}, the map $T_{\mathbb{P}^1} \to S^* \otimes Q$ is given by $4$ polynomials $1,0,0,1$, which cannot vanish simultaneously for any $(x,y)$, so the map \eqref{g24-d=2} is unramified. Now since $S^* \otimes Q \simeq \mathcal{O}(2)^4$, we have the exact sequence
$$0 \rightarrow \mathcal{O}(2) \rightarrow \mathcal{O}(2)^4 \rightarrow N_{\mathbb{P}^1/G(2,4)} \rightarrow 0$$
Hence, the normal bundle $N_{\mathbb{P}^1/G(2,4)} \simeq \mathcal{O}(2)^3$ is balanced.
        \item[(3.2)] $d = 3$: Consider the map 
    \begin{equation}\label{g24-d=3}
[x:y] \longmapsto 
\begin{bmatrix}
x^2 & y^{2} & xy & 0 \\
y & 0 & x & y
\end{bmatrix}
\end{equation}
The rows of the matrix are linearly independent as the minor of columns $2,4$ is a power of $y$, and when plugging in $y=0,$ the minor of columns $1,3$ is a power of $x$.

We claim the module of relations between the columns of \eqref{g24-d=3}, which corresponds to the universal quotient bundle $Q$, is generated by the following relations:
$$\begin{tabular}{c c c c}
$-y^2$ & $x^2$ & 0 & $y^2$\\\hline
0 & $x$ & $-y$ & $x$
\end{tabular}$$
Indeed, it is easy to check that these are relations between the columns of \eqref{g24-d=3}. Moreover, the minor of columns $1,3$ is a power of $y$, and when plugging in $y=0$, the minor of columns $2,4$ is a power of $x$, hence these relations are independent everywhere, and so they generate the module of relations. 

Decompose $Q = Q' \oplus Q''$, where $Q'$ corresponds to the first row and $Q''$ corresponds to the second row. Using \eqref{par-der}, the map $T_{\mathbb{P}^1} \to S^* \otimes Q''$ is given by $2$ polynomials of degrees $1,0$, that is $y, 1$, respectively, which cannot vanish simultaneously for any $(x,y)$, so the map \eqref{g24-d=3} is unramified. Moreover, there is a relation of degree $1$ between these $2$ polynomials: $1\cdot y - y \cdot 1 = 0$, so $(S^* \otimes Q'')/T_{\mathbb{P}^1} \simeq \mathcal{O}(3)$. On the other hand, $S^* \otimes Q' \simeq \mathcal{O}(4) \oplus \mathcal{O}(3)$. Hence, the normal bundle is balanced.

        \item[(3.3)] $d = 4$: Consider the map
        \begin{equation}\label{g24-d=4}
[x:y] \longmapsto 
\begin{bmatrix}
x^3 & y^{3} & x^2y & 0 \\
0 & 0 & x & y
\end{bmatrix}
\end{equation}

The rows of the matrix are linearly independent as the minor of columns $2,4$ is a power of $y$, and the minor of columns $1,3$ is a power of $x$.

We claim the module of relations between the columns of \eqref{g24-d=4}, which corresponds to the universal quotient bundle $Q$, is generated by the following relations:
$$\begin{tabular}{c c c c}
$y^2$ & $0$ & $-xy$ & $x^2$\\\hline
0 & $x^2$ & $-y^2$ & $xy$
\end{tabular}$$
Indeed, it is easy to check that these are relations between the columns of \eqref{g24-d=4}. Moreover, the minor of columns $1,3$ is a power of $y$, and the minor of columns $2,4$ is a power of $x$, hence these relations are independent everywhere, and so they generate the module of relations. 

Decompose $S = S' \oplus S''$, where $S'$ corresponds to the first row and $S''$ corresponds to the second row. Using \eqref{par-der}, the map $T_{\mathbb{P}^1} \to S''^* \otimes Q$ is given by $2$ polynomials of degree $1$, that is $x,y$, which cannot vanish simultaneously when $(x,y) \neq (0,0)$, so the map \eqref{g24-d=4} is unramified. Moreover, there is a relation of degree $2$ between these $2$ polynomials: $y\cdot x - x \cdot y = 0$, so $(S''^* \otimes Q)/T_{\mathbb{P}^1} \simeq \mathcal{O}(4)$. On the other hand, $S^* \otimes Q' \simeq \mathcal{O}(5)^2$. Hence, the normal bundle is balanced.

        \item[(3.4)] $d = 5$: Consider the map
        \begin{equation}\label{g24-d=5}
[x:y] \longmapsto 
\begin{bmatrix}
x^3 & y^{3} & 0 & x^{2}y+xy^2 \\
y^2 & 0 & x^{2} & xy
\end{bmatrix}
\end{equation}

The rows of the matrix are linearly independent as the minor of columns $1,2$ is a power of $y$, and the minor of columns $1,3$ is a power of $x$.

We claim the module of relations between the columns of \eqref{g24-d=5}, which corresponds to the universal quotient bundle $Q$, is generated by the following relations:
$$\begin{tabular}{c c c c}
$0$ & $x^3+x^2y$ & $y^3$ & $-xy^2$\\\hline
$-xy$ & $-xy$ & $-xy+y^2$ & $x^2-xy+y^2$
\end{tabular}$$
Indeed, it is easy to check that these are relations between the columns of \eqref{g24-d=5}. Moreover, the minor of columns $2,4$ is a power of $x$, and when plugging in $x=0$, the minor of columns $3,4$ is a power of $y$, hence these relations are independent everywhere, and so they generate the module of relations. 

Decompose $Q = Q' \oplus Q''$, where $Q'$ corresponds to the first row and $Q''$ corresponds to the second row. Decompose $S = S' \oplus S''$, where $S'$ corresponds to the first row and $S''$ corresponds to the second row. Similar to case 2, using \eqref{par-der}, the map $T_{\mathbb{P}^1} \to S''^* \otimes Q' \oplus S^* \otimes Q''$ is given by $3$ polynomials of degrees $3,3,2$, that is $q_{21}=-x^3-x^2y+xy^2+y^3, q_{12}=xy^2, q_{22}=-x^2+xy+y^2$, respectively, which cannot vanish simultaneously when $(x,y) \neq (0,0)$, so the map \eqref{g24-d=5} is unramified. We have the following relations among these polynomials: 
\begin{align*}
\begin{tabular}{c | c | c }
    $q_{21}$ & $q_{12}$ & $q_{22}$ \\\hline
    $x$ & $2x+y$ & $-x^2-2xy$ \\
    $-y$ & $x+y$ & $y^2+xy$ 
\end{tabular}
\end{align*}

These $2$ relations are independent because the minor of columns $1,2$ is $x^2+3xy+y^2$ and the minor of columns $1,3$ is $xy^2$, which can both vanish only if $(x,y)=(0,0)$. Thus, we have $2$ independent relations of degree $4$, which gives $(S''^* \otimes Q' \oplus S^* \otimes Q'')/T_{\mathbb{P}^1} \simeq \mathcal{O}(6)^2$. On the other hand, $S'^* \otimes Q' \simeq \mathcal{O}(6)$. Hence, the normal bundle is balanced.
    \end{itemize}
\end{case}
\end{proof}
\subsection{$G(2,5)$}

\begin{thm}\label{g(2,5)}
In any characteristic, the normal bundle of a general rational curve of any degree $d$ in the Grassmannian $G(2,5)$ is balanced.
\end{thm}

\begin{proof}
The theorem has been proven for $d = 1$. Write $d = 2t + r$, $r \in \{0,1\}$. We will divide into the 2 following cases:

\setcounter{case}{0}
\begin{case} $r = 0$
\normalfont

Since $G(2,5) \simeq G(3,5)$, we will find a rational curve in $G(3,5)$ whose normal bundle is balanced.
Consider the map
\begin{equation}\label{g25-1}
[x:y] \longmapsto 
\begin{bmatrix}
x^{t} & y^{t} & x^{t-1}y & xy^{t-1} & x^{t}\\
0 & x^{t-1} & y^{t-1} & x^{t-1} & 0 \\
0 & x & 0 & y & x 
\end{bmatrix}
\end{equation}

The rows of the matrix are linearly independent as the minor of columns $1,2,5$ is a power of $x$, and when plugging in $x=0$, the minor of columns $2,3,4$ is a power of $y$.

We claim the module of relations between the columns of \eqref{g25-1}, which corresponds to the universal quotient bundle $Q$, is generated by the following relations:

\begin{center}
    
\begin{tabular}{c c c c c}
$-xy^{t-1}+y^t$ & $-x^{t}$ &  $0$ & $x^{t}$ & $x^t - x^{t-1}y $\\[1pt]
\hline
\\[-6pt]
$x^{t-1}y-xy^{t-1}$ & $-x^2\frac{x^{t-1}-y^{t-1}}{x-y}$ &  $-x^{t}$ & $x\frac{x^{t}-y^{t}}{x-y}$ & $x^t-xy^{t-1}-y^t$
\end{tabular}
\end{center}

Indeed, it is easy to check that these are relations between the columns of \eqref{g25-1}. Moreover, when plugging in $(x,y)=(1,0)$, the minor of columns $2,3$ is nonzero, hence these relations are independent. Finally, these $2$ relations have the sum of degrees equal to $d=2t$, so they generate the module of relations. 

Decompose $S = S' \oplus S''$, where $S'$ corresponds to the first 2 rows and $S''$ corresponds to the last row. Using \eqref{par-der}, the map 
$ T_{\mathbb{P}^1} \rightarrow S''^* \otimes Q$ is given by $2$ polynomials of degrees $t-1$ each, say $q_{31}$ and $q_{32}$:
\begin{align*}
    &q_{31} = x^{t-1}\\
    &q_{32} = \frac{x^t-y^t}{x-y}
\end{align*}

Since $\gcd(q_{31},q_{32})=1$, the map \eqref{g25-1} is unramified. We have the following relation of degree $2t-2$:
$$\frac{x^t-y^t}{x-y}q_{31} + (-x^{t-1})q_{32} = 0.$$

This gives $(S''^* \otimes Q)/T_{\mathbb{P}_1} \simeq \mathcal{O}(2t)$. On the other hand, $S'^* \otimes Q \simeq \mathcal{O}(2t)^2 \oplus \mathcal{O}(2t-1)^2$.
One can easily check that the degrees of the summands of $S'^* \otimes Q$ and $(S''^* \otimes Q)/T_{\mathbb{P}_1}$ are within $1$ of each other, hence so are the degrees of the summands of the normal bundle.
\end{case}

\begin{case} $r = 1$
\normalfont

Consider the map
\begin{equation}\label{g25-2}
[x:y] \longmapsto 
\begin{bmatrix}
x^{t} & y^{t} & x^{t-1}y & xy^{t-1} & 0\\
0 & x^{t} & y^{t} & x^{t} & 0 \\
0 & x & 0 & y & x 
\end{bmatrix}
\end{equation}

The rows of the matrix are linearly independent as the minor of columns $1,2,5$ is a power of $x$, and when plugging in $x=0$, the minor of columns $2,3,4$ is a power of $y$.

We claim the module of relations between the columns of \eqref{g25-2}, which corresponds to the universal quotient bundle $Q$, is generated by the following relations:

\begin{center}
    
\begin{tabular}{c c c c c}
$-xy^{t-1}+y^t$ & $-x^{t}$ &  $0$ & $x^{t}$ & $x^t - x^{t-1}y $\\[1pt]
\hline
\\[-6pt]
$x^{t}y-x^2y^{t-1}$ & $-x^2\frac{x^{t}-y^{t}}{x-y}$ &  $-x^{t+1}$ & $x\frac{x^{t+1}-y^{t+1}}{x-y}$ & $x^{t+1}-xy^{t}-y^{t+1}$
\end{tabular}
\end{center}

Indeed, it is easy to check that these are relations between the columns of \eqref{g25-2}. Moreover, when plugging in $(x,y)=(1,0)$, the minor of columns $2,3$ is nonzero, hence these relations are independent. Finally, these $2$ relations have the sum of degrees equal to $d=2t+1$, so they generate the module of relations.

Decompose $S = S' \oplus S''$, where $S'$ corresponds to the first 2 rows and $S''$ corresponds to the last row. Using \eqref{par-der}, the map 
$ T_{\mathbb{P}^1} \rightarrow S''^* \otimes Q$ is given by $2$ polynomials of degrees $t-1,t$, respectively, say $q_{31}$ and $q_{32}$:
\begin{align*}
    &q_{31} = x^{t-1}\\
    &q_{32} = \frac{x^{t+1}-y^{t+1}}{x-y}
\end{align*}

Since $\gcd(q_{31},q_{32})=1$, the map \eqref{g25-2} is unramified. We have the following relation of degree $2t-1$:
$$\frac{x^{t+1}-y^{t+1}}{x-y}q_{31} -x^{t-1} q_{32} = 0.$$

This gives $(S''^* \otimes Q)/T_{\mathbb{P}_1} \simeq \mathcal{O}(2t+1)$. On the other hand, $S'^* \otimes Q \simeq \mathcal{O}(2t)^2 \oplus \mathcal{O}(2t+1)^2$.
One can easily check that the degrees of the summands of $S'^* \otimes Q$ and $(S''^* \otimes Q)/T_{\mathbb{P}_1}$ are within $1$ of each other, hence so are the degrees of the summands of the normal bundle.
\end{case}
\end{proof}
\subsection{$G(2,6)$}
\begin{thm}\label{g26thm}
In characteristic $0$, the normal bundle of a general rational curve of any degree $d \notin \{3,5\}$ in the
Grassmannian $G(2, 6)$ is balanced.
\end{thm}
\begin{proof} 
    The theorem has been proven for $d = 1$. We will show for $d \notin \{1,3,5\}$.
    Write $d = 14t + r$, $r \in \{0,1,\cdots, 13\}$, we divide into $3$ following cases:
    \setcounter{case}{0}
\begin{case} $r = 2p$, where $p \in \{0,1,\dots,6\}$\end{case}
Define $a,b,c,i$ as follows:
\begin{table}[H]
    \centering
    \begin{tabular}{c||c|c|c|c|c|c|c}
        $p$ & $0$ & $1$ & $2$ & $3$ & $4$ & $5$ & $6$ \\\hline\hline
        $a$ & $-1$ & $0$ & $1$ & $2$ & $3$ & $3$ & $4$ \\\hline
        $b$ & $0$ & $1$ & $1$ & $2$ & $3$ & $4$ & $4$ \\\hline
        $c$ & $0$ & $0$ & $1$ & $1$ & $1$ & $2$ & $3$ \\\hline
        $i$ &  $0$ & $0$ & $1$ & $2$ & $2$ & $3$ & $4$\\
    \end{tabular}
\end{table}

Note that $c = 2p - a - b -1$. Consider the map
{\small
\begin{equation}\label{g(2,6)-1}
\begin{aligned}
[x:y] &\longmapsto \\
&\begin{bmatrix}
x^{7t+p} + x^{4t+b+c-i}y^{3t+p-b-c+i} & y^{7t+p} & x^{2t-2c+2i}y^{5t+p+2c-2i} & -x^{6t-c+2i}y^{t+p+c-2i} & y^{7t+p} & -xy^{7t+p-1}\\
y^{7t+p} & $0$ & x^{t-c+i}y^{6t+p-i+c} & -x^{5t+i}y^{2t+p-i} & x^{7t+p-1}y & -x^{7t+p}
\end{bmatrix}
\end{aligned}
\end{equation}}

The rows of the matrix are linearly independent as the minor of the first $2$ columns is a power of $y$, and when plugging in $y=0$, the minor of the first and last columns is a power of $x$.

We claim the module of relations between the columns of \eqref{g(2,6)-1}, which corresponds to the universal quotient bundle $Q$, is generated by the following relations:
{\small
$$\begin{tabular}{c c c c c c}
\multirow{2}{2cm}[1ex]{$^*x^{4t+a+2i-c-p}\newline y^{t-2i+c+p}$} & \multirow{2}{4.2cm}[1ex]{$x^{5t+a-p-2c+3i}y^{p+2c-3i} \newline- x^{3t+a-p+i+1}y^{2t+p-i-1}-y^{5t+a}$} & \multirow{2}{2.2cm}[1ex]{$x^{5t+a}+x^{2t+p-i-1}\newline y^{3t+a-p+i+1}$}& $-x^{3t+a+i-p} y^{2t-i+p}$ & $y^{5t+a}$ & $-x^{3t+a-p+i} y^{2t+p-i}$ \strutA\\\hline
$x^{t-c+i}y^{4t+b+c-i}$ & \multirow{2}{3.6cm}[1ex]{$-x^{5t+b} -xy^{5t+b-1} \newline+ x^{2t-2c+2i}y^{3t+b+2c-2i}$}& $-y^{5t+b}$ & $-x^{2t+p-i} y^{3t+b-p+i}$ & 0 & $-y^{5t+b}$ \strutA\\\hline
0 & 0 & $x^{4t+c}$ & $y^{4t+c}$ & 0 & 0 \strutA\\\hline
0 & 0 & 0 & 0 & $x$ & $y$ \strutA
\end{tabular} $$
}
*Except in the case $p=3$, where $p+2c-3i <0$, the first relation is
{\small
$$\begin{tabular}{c c c c c c}
$x^{4t+2}y^{t}+x^{t+2}y^{4t}$ & \multirow{2}{3cm}[1ex]{$x^{3t+2}y^{2t} +x^{2t+3}y^{3t-1}\newline -x^2y^{5t}-y^{5t+2}$}&  \multirow{2}{1.8cm}[1ex]{$-x^{3t+1}y^{2t+1}\newline-xy^{5t+1}$} & \multirow{2}{2.5cm}[1ex]{$x^{5t+2}+x^{2t+2}y^{3t} \newline+x^{2t}y^{3t+2}$} & $y^{5t+2}$ & $-x^{3t+1} y^{2t+1}-xy^{5t+1}$
\strutA
\end{tabular} $$}

Indeed, it is easy to check that these are relations between the columns of \eqref{g(2,6)-1}. Moreover, when plugging in $(x,y)=(0,1)$, the minor of columns $2,3,4,6$ is nonzero, hence these relations are independent. Finally, these $4$ relations have the sum of degrees equal to $d=14t+r$, so they generate the module of relations.

Decompose $Q = Q' \oplus Q''$, where $Q'$ corresponds to the first 2 rows and $Q''$ corresponds to the last 2 rows. Using \eqref{par-der}, the map 
$ T_{\mathbb{P}^1} \rightarrow S^* \otimes Q''$ is given by $4$ polynomials of degrees $11t+c+p-2,11t+c+p -2, 7t+p-1,7t+p-1$, respectively, say $q_{31}, q_{32}, q_{41}, q_{42}$:
\begin{align*}
    &q_{31} = (4t+c) x^{6t-c+2i-1}y^{5t+p+2c-2i-1}\\
    &q_{32} = (4t+c)x^{5t+i-1}y^{6t+p+c-i-1}\\
    &q_{41} = y^{7t+p-1} \\
    &q_{42} = x^{7t+p-1}
\end{align*}

Since $\gcd(q_{41},q_{42})=1$, the map \eqref{g(2,6)-1} is unramified. We have the following relations among $q_{ij}$'s:
\begin{align*}
\begin{tabular}{c | c | c | c}
    $q_{31}$ & $q_{32}$ & $q_{41}$ & $q_{42}$ \\\hline
    $x^{t+p+c-2i}$ & 0 & 0 & $-(4t+c)y^{5t+p+2c-2i-1}$ \\
    $y^{t-c+i}$ & $-x^{t-c+i}$ & $0$ & $0$ \\
    $0$ & $- y^{t-c+i}$ & $(4t+c)x^{5t+i-1}$ & $0$ 
\end{tabular}
\end{align*}

These $3$ relations are independent because the minor of the first $3$ columns is a power of $x$ and the minor of columns $1,2,4$ is a power of $y$.
We thus have $3$ independent relations of degrees $12t+p+i-2, 12t+p+i-2, 12t+2p+2c-2i-2$ among these $4$ polynomials, which gives $(S^* \otimes Q'')/T_{\mathbb{P}_1} \simeq \mathcal{O}(12t+p+i)^2 \oplus \mathcal{O}(12t+2p+2c-2i)$. On the other hand, $S^* \otimes Q'\simeq \mathcal{O}(12t+p+a)^2 \oplus \mathcal{O}(12t+p+b)^2$. One can easily check that the degrees of the summands of $S^* \otimes Q'$ and $(S^* \otimes Q'')/T_{\mathbb{P}_1}$ are within $1$ of each other, hence so are the degrees of the summands of the normal bundle.

\begin{case}
   $r=2p+1$, where $p \in \{0,1,3,4,5,6\}$ and $t \geq 1$ if $p \in \{0,1\}$
\end{case}

Define $a,c$ as follows:
\begin{table}[H]
    \centering
    \begin{tabular}{c||c|c|c|c|c|c}
        $p$ & $0$ & $1$ & $3$ & $4$ & $5$ & $6$ \\\hline\hline
        $a$ & $0$ & $1$ & $2$ & $3$ & $4$ & $4$ \\\hline
        $c$ & $0$ & $0$ & $2$ & $2$ & $3$ & $4$ \\
    \end{tabular}
\end{table}

Consider the map
{
\begin{equation}\label{g(2,6)-2}
\begin{aligned}
[x:y] \longmapsto 
\begin{bmatrix}
x^{7t+p+1} + x^{4t+a}y^{3t+p-a+1} & y^{7t+p+1} & x^{2t+1}y^{5t+p} & -x^{6t+c+1}y^{t+p-c} & y^{7t+p+1} & -xy^{7t+p}\\
y^{7t+p} & $0$ & x^{t}y^{6t+p} & -x^{5t+c}y^{2t+p-c} & x^{7t+p-1}y & -x^{7t+p}
\end{bmatrix}
\end{aligned}
\end{equation}}

The rows of the matrix are linearly independent as the minor of the first $2$ columns is a power of $y$, and when plugging in $y=0$, the minor of the first and last columns is a power of $x$.

We claim the module of relations between the columns of \eqref{g(2,6)-2}, which corresponds to the universal quotient bundle $Q$, is generated by the following relations:
{\small
$$\begin{tabular}{c c c c c c}
$x^{t}y^{4t+a}$ & \multirow{2}{3.2cm}[1ex]{$-x^{5t+a} +x^{2t+1}y^{3t+a-1}\newline-xy^{5t+a-1}$}& $-y^{5t+a}$ &$x^{2t+p-c} y^{3t+a-p+c}$ & $0$ & $-y^{5t+a}$ \strutA\\[0cm]\hline
$^*$$x^{4t+p-a}y^{t+p-c}$ & \multirow{3}{3.3cm}[1ex]{$x^{5t+p-a+1}y^{p-c-1} \newline +x^{3t+p-a+1}y^{2t+p-c-1}\newline-y^{5t+2p-a-c}$}&  \multirow{2}{1.5cm}[1ex]{$-x^{3t+p-a}\newline y^{2t+p-c}$} & \multirow{2}{3cm}[1ex]{$x^{5t+2p-a-c} \newline+x^{2t+p-c-1}y^{3t+p-a+1}$ } & $y^{5t+2p-a-c}$ & $-x^{3t+p-a} y^{2t+p-c}$
\strutA\\[0.7cm]\hline
0 & 0 & $x^{4t+c}$ & $y^{4t+c}$ & 0 & 0 \strutA\\\hline
0 & 0 & 0 & 0 & $x$ & $y$ \strutA
\end{tabular} $$}

*Except in the case $p=0$, where $p-c-1 <0$, the second relation is
{
$$\begin{tabular}{c c c c c c}
$x^{4t}y^{t}+x^{t+1}y^{4t-1}$ & \multirow{2}{4cm}[1ex]{$-x^{3t+1}y^{2t-1}+x^{2t+2}y^{3t-2}\newline -x^2 y^{5t-2}-y^{5t}$}& $-x^{3t}y^{2t}\newline-xy^{5t-1}$ & $x^{5t}+x^{2t+1}y^{3t-1}$ & $y^{5t}$ & $-x^{3t} y^{2t}-xy^{5t-1}$
\strutA
\end{tabular} $$}

Indeed, it is easy to check that these are relations between the columns of \eqref{g(2,6)-2}. Moreover, when plugging in $(x,y)=(0,1)$, the minor of columns $2,3,4,6$ is nonzero, hence these relations are independent. Finally, these $4$ relations have the sum of degrees equal to $d=14t+r$, so they generate the module of relations.

\begin{itemize}
    \item[(i)] $p \in \{0,1,3,4\}$:
Decompose $Q = Q' \oplus Q''$, where $Q'$ corresponds to the first 2 rows and $Q''$ corresponds to the last 2 rows. Using Equation $\eqref{par-der}$, the map 
$ T_{\mathbb{P}^1} \rightarrow S^* \otimes Q''$ is given by $4$ polynomials of degrees $11t+p+c-1,11t+p+c -2, 7t+p,7t+p-1$, respectively, say $q_{31}, q_{32}, q_{41}, q_{42}$:
\begin{align*}
    &q_{31} = (4t+c) x^{6t+c}y^{5t+p-1}\\
    &q_{32} = (4t+c)x^{5t+c-1}y^{6t+p-1}\\
    &q_{41} = y^{7t+p} \\
    &q_{42} = x^{7t+p-1}
\end{align*}
Since $\gcd(q_{41},q_{42})=1$, the map \eqref{g(2,6)-2} is unramified. We have the following relations among the polynomials:
\begin{align*}
\begin{tabular}{c | c | c | c}
    $q_{31}$ & $q_{32}$ & $q_{41}$ & $q_{42}$ \\\hline
    $x^{t+p-c-1}$ & 0 & 0 & $-(4t+c)y^{5t+p-1}$ \\
    $y^{t}$ & $-x^{t+1}$ & $0$ & $0$ \\
    $0$ & $- y^{t+1}$ & $(4t+c)x^{5t+c-1}$ & $0$ 
\end{tabular}
\end{align*}
These $3$ relations are independent because the minor of the first $3$ columns is a power of $x$ and the minor of columns $1,2,4$ is a power of $y$.
We thus have $3$ independent relations of degrees $12t+2p-2, 12t+p+c-1, 12t+p+c-1$ among these $4$ polynomials, which gives $(S^* \otimes Q'')/T_{\mathbb{P}_1} \simeq \mathcal{O}(12t+p+c+1)^2 \oplus \mathcal{O}(12t+2p)$. On the other hand, $S^* \otimes Q'\simeq \mathcal{O}(12t+p+a)\oplus \mathcal{O}(12t+p+a+1)\oplus \mathcal{O}(12t+3p-a-c)\oplus\mathcal{O}(12t+3p-a-c+1)$. One can easily check that the degrees of the summands of $S^* \otimes Q'$ and $(S^* \otimes Q'')/T_{\mathbb{P}_1}$ are within $1$ of each other, hence so are the degrees of the summands of the normal bundle.

    \item[(ii)]$p \in \{5,6\}$:
    Decompose $S = S_1 \oplus S_2$, where $S_1$ and $ S_2$ correspond to the first and the second row of the matrix \eqref{g(2,6)-2}, respectively.
Decompose $Q = Q_1 \oplus Q_2\oplus Q_{34}$, where $Q_1$, $Q_2$, $Q_{34}$ correspond to the first row, the second row, and the last 2 rows, respectively. We derive a variant of the exact sequence \eqref{exact-seq}:
\begin{align*}
    0 \rightarrow S^* \otimes Q_1 \oplus S_1^* \otimes Q_2 \rightarrow N_{\mathbb{P}^1/G(k,n)} \rightarrow  (S_2^* \otimes Q_2 \oplus S^* \otimes Q_{34}) / {T_{\mathbb{P}^1}}  \rightarrow 0
\end{align*}

Note that $a+c=p+2$. Using Equation $\eqref{par-der}$, the map 
$ T_{\mathbb{P}^1} \rightarrow S^* \otimes Q''$ is given by $6$ polynomials of degrees $12t+3p-a-c-2= 12t+2p-4, 11t+p+c-1,11t+p+c-2, 7t+p,7t+p-1$, respectively, say $q_{21},q_{22}, q_{31}, q_{32}, q_{41}, q_{42}$. 
{
\begin{align*}
&q_{22} = (2t+p-c)x^{10t+2p-a-1}y^{2t+a-3} + (2t+p-c-1)x^{7t+p-2}y^{5t+p-2} - tx^{4t+p-a-1}y^{8t+p+a-3}\\
    &q_{31} = (4t+c) x^{6t+c}y^{5t+p-1}\\
    &q_{32} = (4t+c)x^{5t+c-1}y^{6t+p-1}\\
    &q_{41} = y^{7t+p} \\
    &q_{42} = x^{7t+p-1}
\end{align*}
}
Since $\gcd(q_{41},q_{42})=1$, the map \eqref{g(2,6)-2} is unramified. We have the following relations among the polynomials:
{\small
\begin{align*}
\begin{tabular}{c|c | c | c | c}
    $q_{22}$ & $q_{31}$ & $q_{32}$ & $q_{41}$ & $q_{42}$ \\\hline $x$ & 0 & 0  & $tx^{4t+p-a}y^{t+a-3}$ & $-(2t+p-c)x^{3t+p-a+1}y^{2t+a-3} - (2t+p-c-1)y^{5t+p-2}$ \\
    $y$ & $-\frac{2t+p-c-1}{4t+c}x^{t+p-c-2} $ & 0 & $tx^{4t+p-a-1}y^{t+a-2}$ & $-(2t+p-c)x^{3t+p-a}y^{2t+a-2}$ 
    \\
    0  & $-y^{t}$ & $x^{t+1}$ & $0$ & $0$ \\
    0 & $0$ & $y^{t+1}$ & $-(4t+c)x^{5t+c-1}$ & $0$ 
\end{tabular}
\end{align*}}

These $4$ relations are independent because the minor of the first $4$ columns is a power of $x$ and the minor of columns $1,2,3,5$ is a power of $y$. We thus have $4$ independent relations of degrees $12t+2p-3,12t+2p-3, 12t+p+c-1, 12t+p+c-1$ among these $4$ polynomials, which gives $$(S_2^* \otimes Q_2 \oplus S^* \otimes Q_{34}) / {T_{\mathbb{P}^1}} \simeq \mathcal{O}(12t+2p-1)^2 \oplus \mathcal{O}(12t+p+c+1)^2.$$
On the other hand, 
\begin{align*}
    S^* \otimes Q_1 \oplus S_1^* \otimes Q_2 &\simeq \mathcal{O}(12t+p+a+1) \oplus \mathcal{O}(12t+p+a) \oplus \mathcal{O}(12t+3p-a-c+1)^2 \\ &\simeq \mathcal{O}(12t+p+a+1) \oplus \mathcal{O}(12t+p+a) \oplus \mathcal{O}(12t+2p-1)^2.
\end{align*}
One can easily check that the degrees of the summands of $S^* \otimes Q_1 \oplus S_1^* \otimes Q_2$ and $(S_2^* \otimes Q_2 \oplus S^* \otimes Q_{34}) / {T_{\mathbb{P}^1}}$ are within $1$ of each other, hence so are the degrees of the summands of the normal bundle.
\end{itemize}

\begin{case}
    $r = 5$, $t \geq 1$
\end{case}
Consider the map
\begin{equation}\label{g(2,6)-3}
\begin{aligned}
[x:y] \longmapsto 
\begin{bmatrix}
x^{7t+3} + x^{4t+2}y^{3t+1} & y^{7t+3} & x^{2t-1}y^{5t+4} & -x^{6t+1}y^{t+2} & y^{7t+3} & -xy^{7t+2}\\
y^{7t+2} + x^{6t+1}y^{t+1} & 0 & x^{t-1}y^{6t+3} & -x^{5t+1}y^{2t+1} & x^{7t+1}y & -x^{7t+2}
\end{bmatrix}
\end{aligned}
\end{equation}

The rows of the matrix are linearly independent as the minor of the first $2$ columns is a power of $y$, and when plugging in $y=0$, the minor of the first and last columns is a power of $x$.

We claim the module of relations between the columns of \eqref{g(2,6)-3}, which corresponds to the universal quotient bundle $Q$, is generated by the following relations:
{\small
$$\begin{tabular}{c c c c c c}
$x^{4t-1}y^{t+2}$ & $x^{5t-1}y^{2} - x^{3t+1}y^{2t} + x^{3t-1}y^{2t+2} -y^{5t+1}$ & $-x^{3t}y^{2t+1}$ & $x^{5t+1}+x^{2t}y^{3t+1}$ & $y^{5t+1}$ & $-x^{3t}y^{2t+1}+x^{3t-2}y^{2t+3}$\strutA\\
\hline
\multirow{2}{1.5cm}[1ex]{$x^{4t-2}y^{t+3}\newline +x^{t-1}y^{4t+2}$} & \multirow{2}{4.5cm}[1ex]{$-x^{5t+1}+x^{5t-2}y^{3} - x^{3t}y^{2t+1} \newline+ x^{3t-2}y^{2t+3} +x^{2t-1}y^{3t+2}-xy^{5t}$} & \multirow{2}{1.6cm}[1ex]{$-x^{3t-1}y^{2t+2} \newline-y^{5t+1}$} & \multirow{2}{2cm}[1ex]{$x^{5t}y+x^{2t+1} y^{3t} \newline+ x^{2t-1} y^{3t+2}$} & 0 & \multirow{2}{2.7cm}[1ex]{$-x^{3t-1} y^{2t+2}\newline +x^{3t-3} y^{2t+4}-y^{5t+1}$} \strutA\\[0cm]\hline
0 & 0 & $x^{4t+2}$ & $y^{4t+2}$ & 0 & 0 \strutA\\ \hline
0 & 0 & 0 & 0 & $x$ & $y$\strutA
\end{tabular}$$
}

Indeed, it is easy to check that these are relations between the columns of \eqref{g(2,6)-3}. Moreover, when plugging in $(x,y)=(0,1)$, the minor of columns $2,3,4,6$ is nonzero, hence these relations are independent. Finally, these $4$ relations have the sum of degrees equal to $d=14t+5$, so they generate the module of relations.

Decompose $S = S' \oplus S''$, where $S'$ and $S''$ corresponds to the first and the second row of the matrix \eqref{g(2,6)-2}, respectively.
Decompose $Q = Q' \oplus Q''$, where $Q'$ corresponds to the first 2 rows and $Q''$ corresponds to the last 2 rows. We derive a variant of the exact sequence \eqref{exact-seq}:
\begin{align*}
    0 \rightarrow S'^* \otimes Q' \rightarrow N_{\mathbb{P}^1/G(k,n)} \rightarrow  (S''^* \otimes Q' \oplus S^* \otimes Q'') / {T_{\mathbb{P}^1}}  \rightarrow 0
\end{align*}

Using Equation \eqref{par-der}, the map 
$ T_{\mathbb{P}^1} \rightarrow S''^* \otimes Q' \oplus S^* \otimes Q''$ is given by $6$ polynomials of degrees $12t+1, 12t+1, 11t+3,11t+2, 7t+2,7t+1$, respectively, say $q_{12},q_{22}, q_{31}, q_{32}, q_{41}, q_{42}$. 
{\small
\begin{align*}
    &q_{12} = (2t+1)x^{10t+1}y^{2t} - (t+1)x^{10t-1}y^{2t+2}+ (2t)x^{7t}y^{5t+1} - (t-1)x^{4t-2}y^{8t+3}\\
    &q_{22} = (2t+1)x^{10t}y^{2t+1} - (t+1)x^{10t-2}y^{2t+3} + (2t+1)x^{7t+1}y^{5t} + tx^{7t-1}y^{5t+2} - (t-1)(x^{4t-3}y^{8t+4} +x^{t-2}y^{11t+3})\\
    &q_{31} = (4t+2) x^{6t}y^{5t+3}\\
    &q_{32} = (4t+2)x^{5t}y^{6t+2}\\
    &q_{41} = y^{7t+2} \\
    &q_{42} = x^{7t+1}
\end{align*}
}
Since $\gcd(q_{41},q_{42})=1$, the map \eqref{g(2,6)-3} is unramified. We have the following relations among these $6$ polynomials:
{\small
\begin{align*}
\begin{tabular}{c | c | c | c | c | c}
    $q_{12}$ & $q_{22}$ &$q_{31}$ & $q_{32}$ & $q_{41}$ & $q_{42}$ \\\hline
    $x$ & $0$ & $0$ & $0$ & $(t-1)x^{4t-1}y^{t+1}$ & $-(2t+1)x^{3t+1}y^{2t} + (t+1)x^{3t-1}y^{2t+2} - (2t)y^{5t+1} $  \\
    $-y$ & $2x$ & $0$ & $0$ & $(t-1)(x^{4t-2}y^{t+2} + 2x^{t-1}y^{4t+1})$ & $- (2t+1)x^{3t}y^{2t+1} + (t+1)x^{3t-2}y^{2t+3} - 2(2t+1)xy^{5t}$ \\
     $0$ & $y$ & $-\frac{t}{4t+2}x^{t-1}$ & $0$ & $(t-1)(x^{4t-3}y^{t+3} + x^{t-2}y^{4t+2})$ & $- (2t+1)x^{3t-1}y^{2t+2} + (t+1)x^{3t-3}y^{2t+4} - (2t+1)y^{5t+1}$ \\
     $0$ & $0$ & $ -y^{t-1}$ & $x^t$ & $0$ & $0$ \\
     $0$ & $0$ & $0$ & $y^t$ & $-(4t+2)x^{5t}$ & $0$
\end{tabular}
\end{align*}
}
These $5$ relations are independent because the minor of the first $5$ columns is a power of $x$, and when plugging in $x = 0$,  the minor of columns $1,2,3,4,6$ is a power of $y$. Thus, we have $5$ independent relations of degree $12t+2$, which gives $(S''^* \otimes Q' \oplus S^* \otimes Q'') / {T_{\mathbb{P}^1}} \simeq \mathcal{O}(12t+4)^5$. On the other hand, $S'^* \otimes Q' \simeq \mathcal{O}(12t+4)^2$.
One can easily check that the degrees of the summands of $S'^* \otimes Q'$ and $(S''^* \otimes Q' \oplus S^* \otimes Q'') / {T_{\mathbb{P}^1}}$ are equal, hence so are the degrees of the summands of the normal bundle.
\end{proof}
\section{Comparison with CLV24}
In this section, we describe the inductive argument in \cite{CLV24} that is used to show the 2-balancedness of normal bundles and see where it fails in proving the balancedness of normal bundles. 

In \cite{CLV24}, the authors showed that the normal bundle of a general curve is $2$-balanced using modifications of vector bundles. They established that the $k$-balancedness of the normal bundle of the general curve of degree $d+1$ follows from the $k$-balancedness of the normal bundle of the general curve of degree $d$ specialized under a one-secant degeneration. Then, iteratively peeling off one-secant lines from the general curve yields a curve whose normal bundle is $2$-balanced. Here we make precise the definitions, attempt this strategy to show the normal bundle of a general curve in $G(2,n)$ is balanced, and compare with our approach. 

\begin{definition}[Modifications of vector bundles] Let $X$ be a scheme and $D$ be a Cartier divisor on $X$. If $E$ is a vector bundle and $F$ is a subbundle of $E$, then we define the negative modification
$$E[D \xrightarrow{-} F] = \ker(E \to (E/F)|_D),$$
and the positive modification
$$E[D \xrightarrow{} F] = E[D \xrightarrow{-} F](D).$$
\end{definition}
\begin{definition}[Generalized pointing bundles] 
Let $V$ be an $(a+b)$-dimensional vector space and consider the Grassmannian $G(a,V)$. Let $p \in \mathbb{P}V$ be a general point and $h$ be a general codimension 1 subspace of $V$. Let $U_p := \{\lambda \in G(a,V) \mid p \notin \lambda\}$ and $U^h := \{\lambda \in G(a,V) \mid \lambda \not\subset h\}$. 
There is a projection map $\pi_p: U_p \to G(a,V/p)$ sending $\lambda$ to its image in $V/p$, and an intersection map $\pi^h: U^h \to G(a-1,h)$ sending $\lambda \mapsto \lambda \cap h$.

The lower generalized pointing bundle is defined as
$N_{C \to p} := \ker(N_C \to N_{\pi_p(C)})$. The upper generalized pointing bundle is defined as
$N^{C \to h} := \ker(N_C \to N_{\pi^h(C)})$.
\end{definition}

\begin{cor}[Corollary 3.6 in \cite{CLV24}]
    Suppose $C \subset U_p \cap U^h$. Then $N_{C \to p} \cap N^{C \to h} \simeq \mathcal{O}_C$.
\end{cor}
\begin{cor}[Corollary 3.7 in \cite{CLV24}]\label{cor-general}
    Fix $x \in C$ and p. For general $h$, the fiber $(N_{C \to p} \cap N^{C \to h})|_x$ is general in $N_{C \to p}|_x$.
\end{cor}
Define notations for modifications towards generalized pointing bundles
$$N_C[x \rcurvearrowright p] := N_C[x \to N_{C \to p}] \text{ and } N_C[x \curvearrowright h] := N_C[x \to N^{C \to h}] $$

Let $L$ be a line in $G(a,V)$. Let $\mathcal{C} \subset G(a,V) \times \Delta \to \Delta$ be a family of curves with smooth total space, whose central fiber $\mathcal{C}_0$ is a reducible curve $C \cup_x L \subset G(a,V)$, where $C$ meets $L$ quasi-transversely at a single point $x \in G(a,V)$. Denote by $\mathcal{N}$ the normal bundle of $\mathcal{C}$ in $G(a,V) \times \Delta$. Let $\mathcal{N}'$ be a modification of $\mathcal{N}$ along Cartier divisors $\{D_i\}$ that do not meet $L$. Applying corresponding modifications to $N_C$ gives $N_C'$. Since $L$ is a $(-1)$-curve in the fiber of $\mathcal{C} \to \Delta$, we can blow down $L$ to obtain a family $\mathcal{C}^- \to \Delta$ whose central fiber is $C$. 
\begin{prop}\label{mainprop}[Lemma 4.1 in \cite{CLV24}]
    Let $p$ be a point and $h$ be a hyperplane in $\mathbb{P}V$ such that every $a$-plane in $L$ contains $x \cap h$ and is contained in $x + p$. There exists a vector bundle on $\mathcal{C}^-$ whose general fiber agrees with $N'|_{\Delta^*}$ and whose special fiber is $N_C'[x \rcurvearrowright p][x \curvearrowright h]$.
\end{prop}

\begin{lemma}\label{deg-exactseq-balanced}
        Let $0 \to E \to F \to G \to 0$ be an exact sequence, where $E$ and $G$ are balanced. 
        \begin{itemize}
            \item[(i)] If $\mu(E) \leq \mu(F)$, then each summand of $G$ has degree at least $\lfloor \mu(E) \rfloor$ (hence so does each summand of $F$ by Lemma \ref{deg-exact-seq}).
            \item[(ii)] If $\mu(E) \geq \mu(F)$, then each summand of $G$ has degree at most $\lceil \mu(E) \rceil$ (hence so does each summand of $F$ by Lemma \ref{deg-exact-seq}).
        \end{itemize}
    \end{lemma}
    \begin{proof}
        Since $\mu(F)$ is the weighted average of $\mu(E)$ and $\mu(G)$, we have $\mu(E) \leq \mu(F) \Leftrightarrow \mu(E) \leq \mu(G)$, in which case each summand of $G$ has degree at least $\lfloor \mu(G)\rfloor \geq \lfloor \mu(E)\rfloor$. In case $(ii)$, we similarly have $\mu(E) \geq \mu(G)$, so each summand of $G$ has degree at most $\lceil \mu(G)\rceil \leq \lceil \mu(E)\rceil.$
    \end{proof}

    Let $C$ be a general rational curve of degree $d$ in $G(2,b+2)$ and $n$ be a nonnegative integer. For general points $x_i \in C$ and $p_i \in \mathbb{P}^{1+b}$, let $N':= N_C[x_1\rcurvearrowright p_1]\dots[x_n\rcurvearrowright p_n]$.
    Let $\delta = \frac{(b+2)d-2+2n}{2b-1}-\frac{d}{2}-n$. Then $\delta \leq d-1$.
    
    When $0 \leq \delta \leq d-1$, we iteratively specialize $C$ to the union of a rational curve of degree $1$ less and a $1$-secant line, then apply Proposition \ref{mainprop}, for $[\delta]$ times, where $[\delta] \in \{\lfloor \delta \rfloor, \lceil\delta\rceil\}$. Then $C$ has degree $d-[\delta]$ and the normal bundle becomes
    $$N_C[y_1 \rcurvearrowright p_1][y_1  \curvearrowright h_1]...[y_{[\delta]} \rcurvearrowright p_{[\delta]}][y_{[\delta]}  \curvearrowright h_{[\delta]}].$$
    Specializing all $p_i$ to a point $p$ and specializing all $h_i$ to contain the point $p$ induces a specialization of $N_C$ to $N_C'$, and we have the exact sequence of projection from $p$:
    \begin{equation}\label{exseq-2}
    0 \to S' \to N_C' \to N' \to 0
\end{equation}
    where $S' := S|_C^*(n+[\delta])[y_1 \curvearrowright \pi_p(h_1)]\dots[y_{[\delta]} \curvearrowright \pi_p(h_{[\delta]})]$, and $N' := N_{\pi_p(C)}[y_1 \curvearrowright \pi_p(h_1)]\dots[y_{[\delta]} \curvearrowright \pi_p(h_{[\delta]})].$
    
    We have $S'$ is balanced by Corollary \ref{cor-general}, with slope $\mu(S
') = \frac{d}{2}+n+[\delta]$, and $N' \subset G(b-1,b+1)$ has slope $\mu(N') = \frac{(b+1)d-2[\delta]-2}{2b-3}$, and $N_C'$ has slope $\mu(N_C') = \frac{(b+2)d-2+2n}{2b-1}$.
We have $\frac{d}{2}+n+\lfloor\delta\rfloor \leq \frac{d}{2}+n+\delta = \mu(N_C') < \frac{d}{2}+n+\lceil\delta\rceil$. Assume $N'$ is balanced. Then 
        \begin{itemize}
        \item[(i)] Choosing $[\delta] = \lfloor\delta\rfloor$ gives $\mu(S') \leq \mu(N_C')$, so by Lemma \ref{deg-exactseq-balanced}, each summand of $N_C'$ has degree at least $\lfloor\frac{d}{2}\rfloor + n + \lfloor\delta\rfloor $.
        \item[(ii)] Choosing $[\delta] = \lceil\delta\rceil$ gives $\mu(S') \geq \mu(N_C')$, so by Lemma \ref{deg-exactseq-balanced}, each summand of $N_C'$ has degree at most $\lceil\frac{d}{2}\rceil + n + \lceil\delta\rceil $.
        \end{itemize}
    
    Hence all summands of $N_C'$ have degree in the interval $[\lfloor\frac{d}{2}\rfloor + n + \lfloor\delta\rfloor ,\lceil\frac{d}{2}\rceil + n + \lceil\delta\rceil]$. When $2\mid d$, this interval has length 1, therefore $N_C'$ is balanced. However, when $2 \nmid d$, we cannot conclude that $N_C'$ is balanced. We then consider the exact sequence \eqref{exseq-2}, which can be written as
    \begin{equation}\notag
    0 \to \mathcal{O}\left(\frac{d-1}{2}+n+[\delta]\right) \oplus \mathcal{O}\left(\frac{d+1}{2}+n+[\delta]\right) \to N_C' \to \mathcal{O}(\lfloor \mu(N')\rfloor)^{r_1} \oplus \mathcal{O}(\lfloor \mu(N')\rfloor+1)^{r_2} \to 0
\end{equation}
We can conclude that $N_C'$ is balanced if either $\frac{d-1}{2}+n+[\delta] = \lfloor \mu(N')\rfloor$ or $\frac{d+1}{2}+n+[\delta] = \lfloor \mu(N')\rfloor$ and $2b-3 \mid (b+1)d-2[\delta]-2$ (that is $N'$ is perfectly balanced). The former is equivalent to 
$$\frac{d-1}{2}+n+[\delta] \leq  \mu(N') < \frac{d+1}{2}+n+[\delta], $$
which is equivalent to
$$\delta - \frac{2b-3}{2(2b-1)} \leq [\delta] < \delta + \frac{2b-3}{2(2b-1)}.$$
Since $\{\delta\} = \frac{k}{2(2b-1)}$ for some odd integer k, the inequality fails for both $\lfloor\delta\rfloor$ and $\lceil\delta\rceil$ iff $2b-3 < k < 2(2b-1)-(2b-3)$, hence iff $k = 2b-1$, that is $\{\delta\} = \frac{1}{2}$. Similarly, the latter also fails iff $\{\delta\} = \frac{1}{2}$. We have $\{\delta\} = \frac{1}{2} \Leftrightarrow 5d-4 \equiv 2b-1 \text{ mod } 2(2b-1) \Leftrightarrow 5d \equiv 2b+3 \text{ mod } 2(2b-1)$ . Therefore, the argument that $N_C'$ is balanced follows from $N'$ is balanced fails iff $d$ is odd and $5d \equiv 2b+3 \text{ mod } 2(2b-1)$, hence iff $5d \equiv 2b+3 \text{ mod } 2(2b-1)$.
  
 Note that in the other congruence classes where this argument works, that is, $5d \not\equiv 2b+3 \text{ mod } 2(2b-1)$, it does not directly prove that the normal bundle is balanced, but it reduces from normal bundle of a curve of degree $d$ with $n$ modifications in $G(2,b+2)$ to normal bundle of a curve of degree $d - [\delta]$ with $[\delta]$ modifications in $G(b-1,b+1)$. We further observe that the cases in which this strategy breaks down are precisely the cases in which we have to consider separately in our approach.
  In $G(2,4)$, the argument in \cite{CLV24} does not work for rational curves of degrees $d=6t+5$, mirroring how we consider the case of degree 5 mod 6 separately in the argument for Theorem \ref{g24thm}. In $G(2,6)$, the argument in \cite{CLV24} does not work for rational curves of degrees $d=14t+5$, mirroring how we consider the case of degree 5 mod 14 separately in the argument for Theorem \ref{g26thm}. 
  
\bibliographystyle{alpha}
\bibliography{citations}
\end{document}